\documentclass[11pt,reqno]{amsart}

\usepackage{amssymb}
\usepackage{amsthm}

\usepackage{color}

\usepackage[latin1]{inputenc}
\usepackage[english]{babel}

\newtheorem {theorem} {Theorem}

\newtheorem {proposition} [theorem]{Proposition}

\newtheorem{example}[theorem]{Example}

\begin{document}
\title[Some Remarks on a Generalized Vector Product]{Some Remarks on a Generalized Vector Product}
\author[P. Acosta-Hum\'anez]{Primitivo B. Acosta-Humánez}
\author[M. Aranda]{Moisés Aranda}
\author[R. N\'uñez]{Reinaldo Núñez}
\maketitle

\noindent\textbf{\footnotesize{Abstract.}} \footnotesize{In this paper we use a generalized
vector product to construct an exterior form $\wedge :\left( \mathbb{R}^{n}\right) ^{k}\rightarrow
\mathbb{R}^{\binom{n}{k}}$, where $\binom{n}{k}=\frac{n!}{(n-k)!k!}$, $k\leq n$. Finally, for $n=k-1$ we introduce the reversing operation to study this generalized vector product over palindromic and antipalindromic vectors.}\\



\section*{Introduction}

It is well known that the vector product over $\mathbb{R}^{3}$ is an alternating $2$-linear function from $\mathbb{R}^{3}\times\mathbb{R}^{3}$ onto $\mathbb{R}^{3}$. Although this vector product is a natural topic to be studied in any course of basic linear algebra, there is a plenty of textbooks on this subject in where it is not considered over $\mathbb{R}^{n}$. The following definition, with interesting remarks, can be found also in \cite{Aranda-Nunez,Marmolejo,Olivert}. Let $$A_{1}=\left( a_{11},a_{12},\ldots ,a_{1n}\right) ,\ldots,\, A_{n-1}=\left( a_{\left( n-1\right) 1},a_{\left( n-1\right) 2},\ldots
,a_{\left( n-1\right) n}\right) $$ be $n-1$ vectors in $\mathbb{R}^{n}$. The vector product over $\mathbb{R}^{n}$ is a function $\times :\left(\mathbb{R}^{n}\right) ^{n-1}\rightarrow\mathbb{R}^{n}$ such that
\begin{equation}\label{vectprod}
\times \left( A_{1},\text{ }A_{2},\text{ }\ldots ,\text{ }A_{n-1}\right)
=A_{1}\times A_{2}\times \cdots \times A_{n-1}=\sum_{k=1}^{n}\left(
-1\right) ^{1+k}\det \left( X_{k}\right) e_{k},
\end{equation} where $e_{k}$ is the $k-$th unity vector of the standard basis of $\mathbb{R}^{n}$ and $X_{k}$ is the square matrix obtained through the deleting of the $k-$th
column of the $\left( a_{ij}\right) _{\left( n-1\right) \times n}$. Notice that in this case the function is not binary and sends a matrix $M$ of size $(n-1)\times n$ to a vector of its $\binom{n}{n-1}$ maximal minors.

 One aim of this paper is to give an algorithm to construct, using elementary techniques, a function with domain in $\left(\mathbb{R}^{n}\right) ^{k}$ and codomain $\mathbb{R}^{\binom{n}{k}}$ which will be an alternating $k$-linear function that obviously generalizes the previous vector product defined over $\mathbb{R}^{n}$.
 
  Using techniques and methods of algebraic geometry we can see that the vector product obtained here, without signs, corresponds to the \emph{Pl\"ucker coordinates} of the matrix $M$, see \cite{harris,hodge}. Although this vector product is known and useful to define the concept of \emph{Grassmanian variety}, see \cite{harris}, we present an alternative construction, avoiding algebraic geometry, which lead us to known results that can be found as for example in \cite{Lang}.
 
Another aim of this work, following \cite{acchro1,acchro2}, is the presentation of some original results concerning to the vector product for $n=k-1$ in palindromic and antipalindromic vectors by means of \emph{reversing operation}.

The way as is presented this paper can allow to students and teachers of basic linear algebra the implementation of these results on their courses, this is our final aim.

\section{A generalized vector product}
In this section we set some preliminaries, properties and the Cramer´s rule as application of the generalized vector product.
\subsection{Preliminaries}
Following \cite{Aranda-Nunez,Marmolejo} we define the generalized vector product over $\mathbb{R}^{n}$ as the function $$\times :\left(\mathbb{R}
^{n}\right) ^{n-1}\rightarrow\mathbb{R}^{n}$$ such that for  $A_{1}=\left( a_{11},a_{12},\ldots ,a_{1n}\right) ,\ldots
,A_{n-1}=\left( a_{\left( n-1\right) 1},a_{\left( n-1\right) 2},\ldots
,a_{\left( n-1\right) n}\right) $, $n-1$ vectors of $\mathbb{R}^{n}$, their vector product is given by
\begin{equation}\label{gvpeq}
\times \left( A_{1},\text{ }A_{2},\text{ }\ldots ,\text{ }A_{n-1}\right)
=A_{1}\times A_{2}\times \cdots \times A_{n-1}=\sum_{k=1}^{n}\left(
-1\right) ^{1+k}\det \left( X_{k}\right) e_{k},
\end{equation}
where $e_{k}$ is the $k-$th element of the canonical basis for $\mathbb{R}^{n}$ and $X_{k}$ is the square matrix obtained after the elimination of the $k-$th column of the matrix $\left( a_{ij}\right) _{\left( n-1\right) \times n}
$. The definition presented in expression \eqref{gvpeq} corresponds to a natural generalization of the vector product of two vectors belonging to $\mathbb{R}^3$.
 \subsection{Some properties}
Let $A_{1},A_{2},\ldots ,A_{n}$ be vectors of $\mathbb{R}^{n}$. The following statements hold.
\begin{enumerate}
\item[1)] $\times \left( A_{1},\text{ \ }A_{2},\text{ }\ldots ,\text{ }%
A_{n-1}\right)$ is an orthogonal vector for the given vectors.

\item[2)] Assume $\alpha,\beta\in \mathbb{R}$, $B_i\in \mathbb{R}^n$:
\begin{eqnarray*}
A_{1}\times A_{2}\times \cdots \times \left( \alpha A_{i}+\beta B_{i}\right)
\times \cdots \times A_{n-1} = \\ A_{1}\times A_{2}\times \cdots \times \alpha
A_{i}\times \cdots \times A_{n-1} +A_{1}\times A_{2}\times \cdots \times \beta B_{i}\times \cdots \times
A_{n-1}.
\end{eqnarray*}

\item[3)] Let the matrix $A$ given by $A= \left( A_{1},A_{2},\ldots ,A_{n}\right)$:
\begin{equation*}
\det A=A_{1}\cdot \left( A_{2}\times
\cdots \times A_{n}\right)=\left( -1\right)
^{1+j}A_{j}\cdot \left( A_{1}\times \cdots \times A_{j-1}\times
A_{j+1}\times A_{n}\right) .
\end{equation*}

\item[4)] The vectors $A_{1},A_{2},\ldots ,A_{n-1}$ are $n-1$ linearly dependent vectors for $
\mathbb{R}^{n}$ if and only if $A_{1}\times A_{2}\times \cdots
\times A_{n-1}=0$.
\end{enumerate}
It is well known that these properties can be proven using the properties of the determinant, see for example \cite{Aranda-Nunez,Marmolejo}.

\subsection{Cramer's rule}

Consider the following system of linear equations
\begin{eqnarray*}
a_{11}x_{1}+a_{12}x_{2}+\cdots +a_{1n}x_{n} &=&b_{1} \\
a_{21}x_{1}+a_{22}x_{2}+\cdots +a_{2n}x_{n} &=&b_{2} \\
&&\vdots  \\
a_{n1}x_{1}+a_{n2}x_{2}+\cdots +a_{nn}x_{n} &=&b_{n}
\end{eqnarray*}
that can be expressed in vectorial way as
\begin{equation}\label{cramer}
x_{1}A_{1}+x_{2}A_{2}+\cdots +x_{n}A_{n}=B,
\end{equation}
being $A_{i}=\left( a_{1i},a_{2i},\ldots ,a_{ni}\right) $ with $i=1,2,\ldots ,n$ and $B=\left( b_{1},b_{2},\ldots ,b_{n}\right) $. Suppose that $\det \left( A_{1},A_{2},\ldots ,A_{n}\right) \neq 0$. For instance such system has unique solution that can be obtained applying the scalar product between the equation \eqref{cramer} and $A_{2}\times A_{3}\times\cdots \times A_{n}$, so we obtain
\begin{eqnarray*}
\left( x_{1}A_{1}+x_{2}A_{2}+\cdots +x_{n}A_{n}\right) \cdot A_{2}\times
A_{3}\times \cdots \times A_{n} &=&B\cdot A_{2}\times A_{3}\times \cdots
\times A_{n} \\
x_{1}A_{1}\cdot A_{2}\times A_{3}\times \cdots \times A_{n} &=&B\cdot
A_{2}\times A_{3}\times \cdots \times A_{n}
\end{eqnarray*}

since $A_{j}\cdot A_{2}\times A_{3}\times \cdots \times A_{n}=0$ for $j=2,3,\ldots ,n$. Therefore
\begin{equation}
x_{1}=\frac{B\cdot A_{2}\times A_{3}\times \cdots \times A_{n}}{A_{1}\cdot
A_{2}\times A_{3}\times \cdots \times A_{n}}=\frac{\det \left(
B,A_{2},A_{3},\cdots ,A_{n}\right) }{\det (A_{1},A_{2},A_{3},\cdots ,A_{n})}.
\end{equation}

In a general way, we can obtain
\begin{eqnarray*}
x_{i} &=&\frac{B\cdot A_{1}\times A_{2}\times \cdots \times A_{i-1}\times
A_{i+1}\times \cdots \times A_{n}}{A_{i}\cdot A_{1}\times A_{2}\times \cdots
A_{i-1}\times A_{i+1}\times \cdots \times A_{n}} \\
&=&\frac{\left( -1\right) ^{i+1}\det \left( A_{1},A_{2},\ldots
,A_{i-1},B,A_{i+1},\cdots ,A_{n}\right) }{\left( -1\right) ^{i+1}\det
(A_{1},A_{2},A_{3},\cdots ,A_{n})} \\
&=&\frac{\det \left( A_{1},A_{2},\ldots ,A_{i-1},B,A_{i+1},\cdots
,A_{n}\right) }{\det (A_{1},A_{2},A_{3},\cdots ,A_{n})},
\end{eqnarray*}
that is, the well-known \emph{Cramer's rule}.

\section{Didactic way to define $\wedge$: algorithm and properties}
In this section we propose a didactic way to define the exterior product $\wedge$. To do this, we set an algorithm to the construction of $\wedge$ and as consequence of this construction arise some properties.
\subsection{Algorithm to the construction of $\wedge$}

Here we present an algorithm and some simple examples to illustrate it.

\subsubsection*{Step 1}

Consider $n\in\mathbb{N}$ and $1\leq k\leq n$, being $k$ an integer. We define
\begin{equation*}
I=\left\{ i_{1}i_{2}\cdots i_{k}:1\leq i_{1}<i_{2}<\cdots <i_{k}\leq
n\right\} ,
\end{equation*} this means that the elements belonging to $I$ are chains of numbers conformed in agreement with the lexicographic order.

\begin{example}\label{ex1}
For $n=5$ and $k=3$ we have
\begin{equation*}
I=\left\{ 123,\text{ }124,\text{ }125,\text{ }134,\text{ }%
135,145,234,235,245,345\right\}.
\end{equation*}
As we can see, $\#I=\binom{n}{k}=\binom{5}{3}=10$.
\end{example}

\begin{example}\label{ex2}
For $n=5$ and $k=2$, we obtain $\binom{5}{2}=10$ and for instance $I$ is given by
\begin{equation*}
I=\left\{ 12,13,14,15,23,24,25,34,35,45\right\} .
\end{equation*}
\end{example}

\subsubsection*{Step 2}

We set that $I$ should be ordered lexicographically.
\begin{equation*}
I_{\left( 1\right) }<I_{\left( 2\right) }<\cdots <I_{\left( \binom{n}{k}%
\right) }
\end{equation*}%
In this way, if $I_{s}\in I$, then there exists $p$ (only one) such that $I_{s}=I_{( p) }$. Thus, we can define $p$ as the \textit{rank} of $I_{s}$ and will be denoted by $r\left( I_{s}\right) =p$. That is, $p$ is the place of $I_{s}$ in $I$ as set of ordered elements lexicographically.

In Example \ref{ex1} we can see that $r\left( 234\right) =7$, $r\left( 345\right) =10$. The same for Example \ref{ex2}, $r\left( 25\right) =7$, $r\left( 35\right) =9$.

\subsubsection*{Step 3}

Let $u_{1}=\left( u_{11},u_{12},\ldots ,u_{1n}\right),  \ldots ,
u_{k}=\left( u_{k1},u_{k2},\ldots ,u_{kn}\right) $, be $k$ vectors of $\mathbb{R}^{n}$, with $k\leq n$. Consider the matrix $U=\left( u_{ij}\right) $ of order $k\times n$ conformed by these vectors. Assume $i_{1}i_{2}\cdots
i_{k}\in I$ and let $U_{i_{1}i_{2}\cdots i_{k}}$ be the matrix of order $k$, conformed by the $k$ columns $i_{1},i_{2},\cdots ,i_{k}$ of $U$. From now on, $U$ always will be a matrix of this kind.

\begin{example}
Consider
\begin{equation*}
U=\left(
\begin{array}{ccccc}
a_{1} & a_{2} & a_{3} & a_{4} & a_{5} \\
b_{1} & b_{2} & b_{3} & b_{4} & b_{5} \\
c_{1} & c_{2} & c_{3} & c_{4} & c_{5}%
\end{array}%
\right) ,
\end{equation*}
in this case, $U_{123}=\left(
\begin{array}{ccc}
a_{1} & a_{2} & a_{3} \\
b_{1} & b_{2} & b_{3} \\
c_{1} & c_{2} & c_{3}%
\end{array}
\right) $ and $U_{245}=\left(
\begin{array}{ccc}
a_{2} & a_{4} & a_{5} \\
b_{2} & b_{4} & b_{5} \\
c_{2} & c_{4} & c_{5}%
\end{array}
\right) $.
\end{example}

Notice that when we choose a particular number of columns of such matrix $U$ exactly corresponds to delete of $U$ the non-selected columns.

\subsubsection*{Step 4}

Consider \begin{equation*}
\left(\mathbb{R}^{n}\right) ^{k}:=\underset{k-\text{times}}{\underbrace{\mathbb{R}^{n}\times\mathbb{R}^{n}\times \cdots \times\mathbb{R}^{n}}}.
\end{equation*}
Now we define the function \textit{exterior product} $\wedge :\left(\mathbb{R}^{n}\right) ^{k}\rightarrow\mathbb{R}^{\binom{n}{k}}$ as follows:
\begin{equation*}
\wedge \left( U\right) =\sum_{i\in I}\left( -1\right) ^{\binom{n}{k}
-r(i)}\det \left( U_{i}\right) e_{\binom{n}{k}-r(i)+1},
\end{equation*}
where $e_{\binom{n}{k}-r(i)+1}$ corresponds to the $\left( \binom{n}{k}-r(i)+1\right) -$th
unity vector of the standard basis of $\mathbb{R}^{\binom{n}{k}}$.

For convenience, we can write
\begin{equation*}
\wedge \left( U\right) =\wedge \left( u_{1},u_{2},\ldots ,u_{k}\right)
=u_{1}\wedge u_{2}\wedge \ldots \wedge u_{k}\text{.}
\end{equation*}

\begin{example}
Consider the vectors $\left( 2,3,-1,5\right),\left( 4,7,2,0\right)\in \mathbb{R}^{4}$. The vector $\left( 2,3,-1,5\right) \wedge \left( 4,7,2,0\right) $ belongs to
$\mathbb{R}
^{\binom{4}{2}}=
\mathbb{R}^{6}$. In this case
\begin{eqnarray*}
I &=&\left\{ 12,13,14,23,24,34\right\} \text{,} \\
U &=&\left(
\begin{array}{cccc}
2 & 3 & -1 & 5 \\
4 & 7 & 2 & 0
\end{array}
\right) ,
\end{eqnarray*}
such that
\begin{eqnarray*}
\wedge \left( U\right) &=&-\left\vert
\begin{array}{cc}
2 & 3 \\
4 & 7%
\end{array}%
\right\vert e_{6}+\left\vert
\begin{array}{cc}
2 & -1 \\
4 & 2%
\end{array}%
\right\vert e_{5}-\left\vert
\begin{array}{cc}
2 & 5 \\
4 & 0%
\end{array}%
\right\vert e_{4}+\left\vert
\begin{array}{cc}
3 & -1 \\
7 & 2%
\end{array}%
\right\vert e_{3}-\left\vert
\begin{array}{cc}
3 & 5 \\
7 & 0%
\end{array}%
\right\vert e_{2} \\
&&+\left\vert
\begin{array}{cc}
-1 & 5 \\
2 & 0%
\end{array}%
\right\vert e_{1} \\
&=&-2e_{6}+8e_{5}+20e_{4}+13e_{3}+35e_{2}-10e_{1} \\
&=&\left( -10,35,13,20,8,-2\right) \text{.}
\end{eqnarray*}
\end{example}

\begin{example}
Consider the canonical basis for $\mathbb{R}^{4}$, that is, $e_{1}=\left( 1,0,0,0\right) $, $e_{2}=\left( 0,1,0,0\right) $, $
e_{3}=\left( 0,0,1,0\right)$ and $e_{1}=\left( 0,0,0,1\right)$. Thus, the exterior product $e_i\wedge e_j$ for $i<j$ is given by
\begin{eqnarray*}
e_{1}\wedge e_{2} &=&-\left( 0,0,0,0,0,1\right)=-e_6\in\mathbb{R}^6, \\
e_{1}\wedge e_{3} &=&\left( 0,0,0,0,1,0\right) =e_5\in\mathbb{R}^6, \\
e_{1}\wedge e_{4} &=&-\left( 0,0,0,1,0,0\right)=-e_4\in\mathbb{R}^6, \\
e_{2}\wedge e_{3} &=&\left( 0,0,1,0,0,0\right)=e_3\in\mathbb{R}^6, \\
e_{2}\wedge e_{4} &=&-\left( 0,1,0,0,0,0\right)=-e_2\in\mathbb{R}^6  , \\
e_{3}\wedge e_{4} &=&\left( 1,0,0,0,0,0\right)=e_1\in\mathbb{R}^6 .
\end{eqnarray*}
As we can see, the set $B=\left\{ e_{1}\wedge e_{2},e_{1}\wedge e_{3},e_{1}\wedge e_{4},e_{2}\wedge
e_{3},e_{2}\wedge e_{4},e_{3}\wedge e_{4}\right\} \subset
\mathbb{R}^{6}$ is one basis for $\mathbb{R}^{6}$.
\end{example}

Notice that in a given basis $B$ for $\mathbb{R}^n$, the exterior product of them taken in sets of $k$-elements without repetition constitutes a basis $B'$ for $\mathbb{R}^{\binom{n}{k}}$.

\subsection{Some properties of $\wedge$}

The following properties are satisfied by $\wedge$:
\begin{enumerate}
\item[1)] If $k=n$, then $\wedge \left( U\right) =\det \left( U\right) $.

\item[2)] If $k=n-1$, then $\wedge $ is the generalized vector product.

\item[3)] If $n$ is even and $k=1$, then $U$ is orthogonal to $\wedge \left(
U\right) $.

\item[4)] $\wedge $ is $k-$linear,
\begin{equation*}
\wedge \left( u_{1},\ldots ,u_{i}+b,\ldots ,u_{k}\right) =\wedge \left(
u_{1},\ldots ,u_{i},\ldots ,u_{k}\right) +\wedge \left( u_{1},\ldots
,b,\ldots ,u_{k}\right) .
\end{equation*}

\item[5)] If $M_{p}$ is a permutation of two rows (being fixed the other ones) of $M$, then $\wedge \left( M_{p}\right) =-\wedge \left(
M\right) $.

\item[6)] If $u_{1},\ldots ,u_{k}$ are $k$ ($\leq n$) linear dependent vectors of $\mathbb{R}^{n}$, then $\wedge \left( u_{1},\ldots
,u_{k}\right) =0\in\mathbb{R}^{\binom{n}{k}}$.
\end{enumerate}

\begin{proof} We proceed according to each item.
\begin{enumerate}
\item[1)]
Assuming $k=n$ we have $\dbinom{n}{k}=\dbinom{n}{n}=1$ \ and $r(i)=1$ (due to $I$ has only one element). For instance
\begin{eqnarray*}
\wedge \left( U\right) &=&\sum_{i\in I}\left( -1\right) ^{\binom{n}{k}%
-r(i)}\det \left( U_{i}\right) e_{\binom{n}{k}-r(i)+1} \\
&=&\det (U_{i}).
\end{eqnarray*}

Trivially we can see that for $\mathbb{R}$, $e_{1}=1$.

\item[2)] Assuming $k=n-1$, we have $\dbinom{n}{k}=\dbinom{n}{n-1}=n$, in this way, $I$ has $n$ elements.
Owing to the symmetry of $\dbinom{n}{k}$, the election of $n-1$ columns of the matrix $U$ corresponds to the elimination of one column of $U$ (precisely the avoided column in the election). In other words, we can see that
\begin{equation*}
U_{i}=X_{n-r(i)+1}
\end{equation*}

where $X_{n-r(i)+1}$ corresponds to the matrix that has been obtained throughout $U$ deleting the ($n-r(i)+1$)-th column such that
\begin{eqnarray*}
\wedge \left( U\right) &=&\sum_{i\in I}\left( -1\right) ^{n-r(i)}\det \left(
U_{i}\right) e_{n-r(i)+1} \\
&=&\sum_{i\in I}\left( -1\right) ^{\left( n-r(i)+1\right) +1}\det \left(
U_{i}\right) e_{n-r(i)+1} \\
&=&\sum_{j=1}^{n}\left( -1\right) ^{j+1}\det \left( X_{j}\right) e_{j} \\
&=&u_{1}\times \ldots \times u_{k}\text{.}
\end{eqnarray*}

\item[3)]  For $n=2p$ and $k=1$, we have $\dbinom{2p}{1}=2p$, thus, the cardinality of $I$ is even and
\begin{equation*}
I=\left\{ 1,2,\ldots ,p,p+1,\ldots ,2p\right\}.
\end{equation*}

Furthermore, $r\left( i\right) =1$. In this way, $\wedge \left( U\right)
\in\mathbb{R}
^{2p}$. On the other hand, considering $U=\left( u_{1},u_{2},\ldots ,u_{2p}\right) $ and $\wedge \left(
U\right) =\left( v_{1},v_{2},\ldots ,v_{2p}\right) $, we obtain
\begin{eqnarray*}
\wedge \left( U\right) &=&\sum_{i\in I}\left( -1\right) ^{2p-i}\det \left(
U_{i}\right) e_{2p-i+1} \\
&=&\sum_{i=1}^{2p}\left( -1\right) ^{i}u_{i}e_{2p-i+1} \\
&=&\left( u_{2p},-u_{2p-1},\ldots ,u_{2},-u_{1}\right),
\end{eqnarray*}

where it follows that $v_{j}=\left( -1\right) ^{j+1}u_{2p-j+1}$ for $j=1,2,\ldots ,2p$. Therefore
\begin{eqnarray*}
U\cdot \wedge \left( U\right) &=&\left( u_{1},u_{2},\ldots
,u_{2p-1},u_{2p}\right) \cdot \left( u_{2p},-u_{2p-1},\ldots
,u_{2},-u_{1}\right) \\
&=&u_{1}u_{2p}-u_{2}u_{2p-1}+\ldots +u_{2p-1}u_{2}-u_{2p}u_{1} \\
&=&\left( u_{1}u_{2p}-u_{2p}u_{1}\right)  +\ldots +(-1)^{p+1}\left(
u_{p}u_{p+1}-u_{p+1}u_{p}\right) \\
&=&0\text{.}
\end{eqnarray*}
\end{enumerate}
Items 4),  5) and 6) can be proven using the properties of the determinant in similar way as the previous ones.
\end{proof}

\section{Reversing operation over $\wedge$}
The \emph{reversing operation} has been applied successfully over rings and vector spaces, see \cite{acchro1,acchro2}. In this section we apply the reversing operation to obtain some results that involve the exterior product with the \emph{palindromic} and \emph{antipalindromic} vectors.
The following results correspond to a generalization of some results presented in \cite{acchro2}.
Consider the matrix $M=\left( m_{i,j}\right) $ of size $m\times n$. The reversing of $M$, denoted by $\overleftarrow{M}$ is given by  $\overleftarrow{M}=(\overleftarrow{m}_{i,j})$, where $\overleftarrow{m}_{i,j}=m_{i,n-j+1}$. We can see that the size of  $\overleftarrow{M}$ is $m\times n$ too. We denote by $J_{n}=\overset{\longleftarrow }{I_{n}}$  the reversing of the identity matrix $I_{n}$ of size $n$. Thus, the following properties can be proven, see \cite{acchro2}.

1. The double reversing:
\begin{equation*}
\overleftarrow{\overleftarrow{M}}=(\overleftarrow{\overleftarrow{m}}_{i,j})=(%
\overleftarrow{m}_{i,n-j+1})=\left( m_{i,n-(n-j+1)+1}\right) =(m_{i,j})=M,
\end{equation*}%

2. $\overleftarrow{M}=MJ_{n}$

3. $J_{n}J_{n}=I_{n}$.
The following definitions were introduced in \cite{acchro2}. A matrix $M$ is called palindromic whether it satisfies $\overleftarrow{M}=M$, in the same way, a matrix $M$ is called antipalindromic whether it satisfies $\overleftarrow{M}=-M$. In particular, for $m=1$, we get palindromic and antipalindromic vectors respectively.

As we can see, the palindromic matrix $M$ satisfies that $m_{i,j}=m_{i,n-j+1}$ and for instance $M$
has at least $\dfrac{n}{2}$ pair of equal columns whether $n$ is even (as well $\dfrac{n}{2}-1$ when $n$ is odd). This fact lead us to the following result.

\begin{proposition}\label{prop1}
$\det (J_{n})=\left\{
\begin{array}{l}
(-1)^{n/2},\, n=2k,\,k\in\mathbb{Z}^+\\
(-1)^{\frac{n+3}{2}},\, n=2k-1,\,k\in\mathbb{Z}^+
\end{array}%
\right. $.
\end{proposition}

\begin{proof} We proceed by induction over $n$.
Assuming $n=1$, we have that $I_{n}=1$ and $J_{n}=1$, thus $\det \left( J_{n}\right) =1=\left(
-1\right) ^{\frac{1+3}{2}}$. Let the proposition be true for $n$, thus we will prove that is also true for $n+1$. We start considering that $n$ is even, so we get
\begin{eqnarray*}
\det \left( J_{n+1}\right)  &=&1\left( -1\right) ^{1+\left( n+1\right) }\det
\left( J_{n}\right)  \\
&=&\left( -1\right) ^{n+2}\left( -1\right) ^{\frac{n}{2}} \\
&=&\left( -1\right) ^{\frac{n}{2}}=\left( -1\right) ^{\frac{(n+1)+3}{2}}.
\end{eqnarray*}
Now, considering $n$ as an positive odd integer, we have
\begin{eqnarray*}
\det \left( J_{n+1}\right)  &=&1\left( -1\right) ^{1+\left( n+1\right) }\det
\left( J_{n}\right)  \\
&=&\left( -1\right) ^{n+2}(-1)^{\frac{n+3}{2}} \\
&=&\left( -1\right) (-1)^{\frac{n+3}{2}} \\
&=&\left( -1\right) ^{\frac{n+5}{2}}=\left( -1\right) ^{\frac{n+1}{2}}.
\end{eqnarray*}
\end{proof}
Now, we study the relationship between the exterior product $\wedge$ and the reversing operation. We start considering $k=n-1$, that is, the generalized vector product over $\mathbb{R}^{n}$. Consider $M_{1}=\left(
m_{11},m_{12},\ldots ,m_{1n}\right) ,\ldots ,M_{n-1}=\left( m_{\left(
n-1\right) ,1},a_{\left( n-1\right) ,2},\ldots ,m_{\left( n-1\right)
,n}\right) ,$ $n-1$ vectors in $\mathbb{R}^{n}$. The generalized vector product is given by the equation \eqref{vectprod}, therefore we obtain
\begin{equation}
\times \left( M_{1},\text{ }M_{2},\text{ }\ldots ,\text{ }M_{n-1}\right)
=\sum_{k=1}^{n}\left(
-1\right) ^{1+k}\det \left( M^{(k)}\right) e_{k},
\end{equation}

being $e_{k}$ the $k$-th element of the canonical basis for $\mathbb{R}^{n}$ and $M^{\left( k\right) }$ is the square matrix obtained after the deleting of the $k$-th column of the matrix $M=\left( m_{ij}\right) _{\left(
n-1\right) \times n}$. The matrix $M^{\left( k\right) }$ is a square matrix of size $\left(
n-1\right) \times \left( n-1\right) $ and is given by
\begin{equation}
M^{(k)}=\left( m_{i,j}^{(k)}\right) =\left\{
\begin{array}{l}
m_{i,j}\text{, si }j<k \\
m_{i,j+1}\text{ si }j\geq k%
\end{array}%
\right. .
\end{equation}

\begin{proposition}\label{prop2}
If we consider $M=\left( m_{ij}\right) _{\left( n-1\right) \times n}$, then $\overleftarrow{M}^{\left( k\right) }=M^{\left( n-k+1\right) }J_{n-1}$, for $1\leq k\leq n$.
\end{proposition}

\begin{proof}
We know that $\overleftarrow{M}=MJ_{n}$, that is, $\left( \overleftarrow{m}_{i,j}\right)
=\left( m_{i,n-j+1}\right)$,  $1\leq j\leq n$. Therefore
\begin{eqnarray*}
\overleftarrow{M}^{\left( k\right) } &=&\left( \overleftarrow{m}%
_{i,j}^{\left( k\right) }\right) =\left\{
\begin{array}{l}
\overleftarrow{m}_{i,j}\text{, si }j<k \\
\overleftarrow{m}_{i,j+1}\text{ si }j\geq k%
\end{array}%
\right.  \\
&=&\left\{
\begin{array}{l}
m_{i,n-j+1}\text{, si }j<k \\
m_{i,n-(j+1)+1}\text{ si }j\geq k%
\end{array}%
\right. .
\end{eqnarray*}
On the other hand,
\begin{equation}
M^{\left( n-k+1\right) }=\left( m_{i,j}^{\left( n-k+1\right) }\right)
=\left\{
\begin{array}{l}
m_{i,j}\text{, si }j<n-k+1 \\
m_{i,j+1}\text{ si }j\geq n-k+1%
\end{array}%
\right. .
\end{equation}
Now, we obtain
\begin{eqnarray*}
M^{\left( n-k+1\right) }J_{n-1} &=&\left( m_{i,(n-1)-j+1}^{\left(
n-k+1\right) }\right) =\left( m_{i,n-j}^{\left( n-k+1\right) }\right)  \\
&=&\left\{
\begin{array}{l}
m_{i,\left( n-j\right) }\text{, si }n-j<n-k+1 \\
m_{i,\left( n-j\right) +1}\text{ si }n-j\geq n-k+1%
\end{array}%
\right.  \\
&=&\left\{
\begin{array}{l}
m_{i,\left( n-j\right) }\text{, si }j>k-1 \\
m_{i,\left( n-j\right) +1}\text{ si }j\leq k-1%
\end{array}%
\right.  \\
&=&\left\{
\begin{array}{l}
m_{i,n-j}\text{, si }j\geq k \\
m_{i,n-j+1}\text{ si }j<k%
\end{array}%
\right. =\overleftarrow{M}^{\left( k\right) }.
\end{eqnarray*}
\end{proof}
The following proposition is a generalization of one result presented in \cite{acchro2}, where was analyzed the reversing of the vector product in $\mathbb{R}^3$.

From now on, for suitability we denote $M=\left( M_{1},M_{2},\ldots ,M_{n-1}\right) $, i.e., $M$ is the matrix that has as rows the vectors $M_{1},$ $M_{2},$ $\ldots ,$ $M_{n-1}
$, thus we obtain $$\overleftarrow{M}=\left( \overleftarrow{M}_{1},\text{ }%
\overleftarrow{M}_{2},\text{ }\ldots ,\text{ }\overleftarrow{M}_{n-1}\right)
.$$ In the same way, for suitability we write $$\mathfrak{M}=\times \left( \overleftarrow{M}_{1},\text{ }\overleftarrow{M}_{2},\text{ }%
\ldots ,\text{ }\overleftarrow{M}_{n-1}\right).$$ 
\begin{proposition}\label{prop3} The generalized vector product of $\overleftarrow{M}_{i}$, being $1\leq i\leq n-1$, satisfies
$$\mathfrak{M}=\left\{
\begin{array}{l}
\left( -1\right) ^{\frac{3n}{2}}\left( \overleftarrow{\times \left( M_{1},%
\text{ }M_{2},\text{ }\ldots ,\text{ }M_{n-1}\right) }\right),\, n=2k,\\
\left( -1\right) ^{\frac{3n+1}{2}}\left( \overleftarrow{\times \left( M_{1},%
\text{ }M_{2},\text{ }\ldots ,\text{ }M_{n-1}\right) }\text{ }\right) ,\, n=2k-1
\end{array}%
\right. ,$$ being $k\in \mathbb{Z}^+$.
\end{proposition}

\begin{proof}
For suitability we denote $M=\left( M_{1},M_{2},\ldots ,M_{n-1}\right) $, i.e., $M$ is the matrix that has as rows the vectors $M_{1},$ $M_{2},$ $\ldots ,$ $M_{n-1}
$, thus we obtain $$\overleftarrow{M}=\left( \overleftarrow{M}_{1},\text{ }%
\overleftarrow{M}_{2},\text{ }\ldots ,\text{ }\overleftarrow{M}_{n-1}\right)
.$$ In the same way, for suitability we write $$\mathfrak{M}=\times \left( \overleftarrow{M}_{1},\text{ }\overleftarrow{M}_{2},\text{ }%
\ldots ,\text{ }\overleftarrow{M}_{n-1}\right).$$ Now, applying the generalized vector product we obtain 
\begin{eqnarray*}
\mathfrak{M} &=&\sum_{k=1}^{n}\left(
-1\right) ^{k+1}\det \left( \overleftarrow{M}^{\left( k\right) }\right) e_{k}
\\
&=&\sum_{k=1}^{n}\left( -1\right) ^{k+1}\det \left( M^{\left( n-k+1\right)
}J_{n-1}\right) e_{k} \\
&=&\sum_{k=1}^{n}\left( -1\right) ^{k+1}\det \left( M^{\left( n-k+1\right)
}J_{n-1}\right) e_{k} \\
&=&\sum_{k=1}^{n}\left( -1\right) ^{k+1}\det \left( M^{\left( n-k+1\right)
}\right) \det \left( J_{n-1}\right) e_{k} \\
&=&\det \left( J_{n-1}\right) \sum_{k=1}^{n}\left( -1\right) ^{n-k}\det
\left( M^{\left( k\right) }\right) e_{n-k+1} \\
&=&\left( -1\right) ^{n+1}\det \left( J_{n-1}\right) \sum_{k=1}^{n}\left(
-1\right) ^{k+1}\det \left( M^{\left( k\right) }\right) e_{n-k+1} \\
&=&\left( -1\right) ^{n+1}\det \left( J_{n-1}\right) \left(
\sum_{k=1}^{n}\left( -1\right) ^{k+1}\det \left( M^{\left( k\right) }\right)
e_{k}\right) J_{n} \\
&=&\left( -1\right) ^{n+1}\det \left( J_{n-1}\right) \left( \overleftarrow{%
\times \left( M_{1},\text{ }M_{2},\text{ }\ldots ,\text{ }M_{n-1}\right) }%
\right)
\end{eqnarray*}

and for instance,%
\begin{eqnarray*}
\mathfrak{M}  &=&\left\{
\begin{array}{l}
\left( -1\right) ^{n+1}(-1)^{\frac{\left( n-1\right) +3}{2}}\left(
\overleftarrow{\times \left( M_{1},\text{ }M_{2},\text{ }\ldots ,\text{ }%
M_{n-1}\right) }\right),\, n=2k\\
\left( -1\right) ^{n+1}\left( -1\right) ^{\frac{n-1}{2}}\left(
\overleftarrow{\times \left( M_{1},\text{ }M_{2},\text{ }\ldots ,\text{ }%
M_{n-1}\right) }\text{ }\right),\, n=2k-1
\end{array}%
\right.  \\
&=&\left\{
\begin{array}{l}
\left( -1\right) ^{\frac{3n}{2}}\left( \overleftarrow{\times \left( M_{1},%
\text{ }M_{2},\text{ }\ldots ,\text{ }M_{n-1}\right) }\right),\, n=2k\\
\left( -1\right) ^{\frac{3n+1}{2}}\left( \overleftarrow{\times \left( M_{1},%
\text{ }M_{2},\text{ }\ldots ,\text{ }M_{n-1}\right) }\text{ }\right), n=2k-1
\end{array}%
\right. .
\end{eqnarray*}
\end{proof}

If $M$ is a palindromic matrix, then the minors $M^{\left( k\right) }$ have at least $\dfrac{n}{2}-1$ pair of equal columns when $n$ is even and respectively $\dfrac{n-1}{2}-1$ when $n$ is odd. This implies that for  $n\geq 4$, the minors have at least one pair of equal columns and for instance $\det \left( M^{\left(
k\right) }\right) $ $=0$ for all $1\leq k\leq n$ and so
\begin{equation}
\times \left( M_{1},\text{ }M_{2},\text{ }\ldots ,\text{ }M_{n-1}\right)
=\mathbf{0}\in\mathbb{R}^{n}.
\end{equation}

This means that the generalized vector product of $(n-1)$ palindromic vectors in  $\mathbb{R}^{n}$ is interesting when $1\leq n\leq 3$. The same result is obtained when we assume $M$ as an antipalindromic matrix.

\section*{Final Remarks}
When we consider the exterior product for $k\neq n-1$, the previous results cannot be applied due to in general they are not true. To illustrate it, we present the following example.
\begin{example}
Consider the vectors $\left( 2,3,-1,5\right) $ and $\left( 4,7,2,0\right)
$ in $\mathbb{R}^{4}$. In this case,
\begin{equation*}
M=\left(
\begin{array}{cccc}
2 & 3 & -1 & 5 \\
4 & 7 & 2 & 0%
\end{array}%
\right) \text{ \ and \ }\overleftarrow{M}=%
\left(\begin{array}{cccc}
5 & -1 & 3 & 2 \\
0 & 2 & 7 & 4%
\end{array}\right)%
\end{equation*}%
As we have seen before, 
\begin{equation*}
\left( 2,3,-1,5\right) \wedge \left( 4,7,2,0\right) =\left(
-10,35,13,20,8,-2\right) .
\end{equation*}
Therefore
\begin{eqnarray*}
\left( 5,-1,3,2\right) \wedge \left( 0,2,7,4\right)  &=& \\
&&-\left\vert
\begin{array}{cc}
5 & -1 \\
0 & 2%
\end{array}%
\right\vert e_{6}+\left\vert
\begin{array}{cc}
5 & 3 \\
0 & 7%
\end{array}%
\right\vert e_{5}-\left\vert
\begin{array}{cc}
5 & 2 \\
0 & 4%
\end{array}%
\right\vert e_{4}+ \\
&&+\left\vert
\begin{array}{cc}
-1 & 3 \\
2 & 7%
\end{array}%
\right\vert e_{3}-\left\vert
\begin{array}{cc}
-1 & 2 \\
2 & 4%
\end{array}%
\right\vert e_{2}+\left\vert
\begin{array}{cc}
3 & 2 \\
7 & 4%
\end{array}%
\right\vert e_{1} \\
&=&-(10)e_{6}+(35)e_{5}-(20)e_{4}+(-7-6)e_{3}-\\ &&(-4-4)e_{2}+(12-14)e_{1} \\
&=&(-2,8,-13,-20,35,-10).
\end{eqnarray*}
\end{example}
Thus, in general, the exterior product does not satisfies
\begin{equation*}
\bigwedge U=\left( -1\right) ^{p}\bigwedge \overleftarrow{U}, \quad \text{for some  }p\in\mathbb{Z}.
\end{equation*}
Finally, although this paper is presented in a didactic way, there are original results corresponding to the relations between the reversing operation and the generalized vector product.
\section*{Acknowledgements}
The first author is partially supported by MICIIN/FEDER grant number MTM2009-
06973 and by Universidad del Norte. The second author is supported by Pontificia Universidad Javeriana. The third author is partially supported by Universidad Sergio Arboleda. The authors thanks to the anonymous referees by their useful comments and suggestions.

{\small \rightline{\sc Primitivo Acosta-Hum\'anez}
\rightline{\sc
Departamento de Matem\'aticas y Estadística}
\rightline{\sc Universidad del Norte} \rightline{\sc Barranquilla, Colombia}}

{\small \smallskip \rightline{{\it e-mail:} \tt pacostahumanez@uninorte.edu.co}}
\bigskip

{\small \rightline{\sc Mois\'es Aranda}
\rightline{\sc
Departamento de Matem\'aticas}
\rightline{\sc Pontificia Universidad
Javeriana} \rightline{\sc Bogot\'a, Colombia}}

{\small \smallskip \rightline{{\it e-mail:} \tt maranda@javeriana.edu.co}}
\bigskip

{\small \rightline{\sc Reinaldo Nunez}
\rightline{\sc Escuela de
Matem\'aticas} \rightline{\sc Universidad Sergio Arboleda}
\rightline{\sc
Bogot\'a, Colombia}}

{\small \smallskip \rightline{{\it e-mail:} \tt reinaldo.nunez@usa.edu.co}%
\label{final} }

\end{document}